\documentclass[10pt,a4paper]{amsart}

\usepackage{amsmath}
\usepackage{amsfonts}
\usepackage{amssymb}
\usepackage{amsthm}
\theoremstyle{plain}
\newcommand{\C}{\mathbb{C}}

\newcommand{\Proj}{\mathbb{P}}
\newcommand{\bpartial}{\bar{\partial}}
\newcommand{\dif}{\textsf{Diff}}

\newcommand{\crit}{\textsf{crit}}

\newcommand{\sing}{\textsf{Sing}}

\newcommand{\cod}{\textsf{cod}}
\newtheorem{theorem}{Theorem}[section]
\newtheorem{theoremno}{Theorem}
\newtheorem{lemma}[theorem]{Lemma}
\newtheorem{proposition}[theorem]{Proposition}
\newtheorem{corollary}{Corollary}
\theoremstyle{definition}
\newtheorem{definition}{Definition}[section]
\newtheorem{example}{Example}[section]
\newtheorem{remark}{Remark}[section]

\date{}
\title{Local normal forms of singular Levi-flat hypersurfaces }
\author{Arturo Fern\'andez-P\'erez}
\address[A. Fern\'andez-P\'erez]{Departamento de Matem\'atica, Universidade Federal de Minas Gerais, UFMG}
\curraddr{Av. Ant\^onio Carlos 6627, 31270-901, Belo Horizonte-MG, Brazil.}
\email{fernandez@ufmg.br}
\author{Gustavo Marra}
\address[Gustavo Marra]{Departamento de Matem\'atica, Universidade Federal de Itajub\'a}
\curraddr{Rua Irm\~a Ivone Drumond 200, 35903-087, Itabira-MG, Brazil.}
\email{marra@unifei.edu.br}

\thanks{This work is partially supported by CNPq Brazil grant number 427388/2016-3}


\begin{document}

\begin{abstract}
We study normal forms of germs of singular real-analytic Levi-flat hypersurfaces. We prove the existence of rigid normal forms for singular Levi-flat hypersurfaces which are defined by the vanishing of the real part of complex quasihomogeneous polynomials with isolated singularity. This result generalizes previous results of Burns-Gong \cite{burnsgong1999} and Fern\'andez-P\'erez \cite{perez2014}. Furthermore, we prove the existence of two new rigid normal forms for singular real-analytic Levi-flat hypersurfaces which are preserved by a change of \textit{isochore coordinates}, that is, a change of coordinates that preserves volume.  

\end{abstract}

\maketitle
\section{Introduction}
In this paper we study normal forms of germs of singular real-analytic Levi-flat hypersurfaces. Our first result is the following. 
\begin{theoremno}\label{teorema1} 
Let $M=\{F=0\}$ be a germ of an irreducible singular real-analytic Levi-flat hypersurface at $0\in\mathbb{C}^2$ such that
\begin{enumerate}
\item[(a)] $F(z)=\mathcal{R}e(Q(z))+H(z,\bar{z})$;
\item[(b)] $Q$ is a complex quasihomogeneous polynomial of quasihomogeneous degree $d$ with isolated singularity at $0\in\mathbb{C}^{2}$.
\item[(c)] $H$ is a germ of real-analytic function at $0\in\mathbb{C}^2$ of order strictly greater than $d$ and $H(z,\bar{z})=\overline{H(\bar{z},z)}$. 
\end{enumerate}
Then there exists a germ of biholomorphism $\phi:(\mathbb{C}^{2},0)\rightarrow(\mathbb{C}^{2},0)$ such that $$\phi(M)=\left\{\mathcal{R}e\left(Q(z)+\sum^{s}_{j=1} c_{j}e_{j}(z)\right)=0\right\},$$ $e_{1},\ldots,e_{s}$ are the elements of the monomial basis of the local algebra of $Q$ of quasihomogeneous degree strictly greater than $d$
and $c_{j}\in\mathbb{C}$.
\end{theoremno}
\par When $M$ is a germ of a singular real-analytic Levi-flat hypersurface at $0\in\mathbb{C}^n$, $n\geq 3$, the same result was proved by Fern\'andez-P\'erez in \cite{perez2011}. Therefore, the above theorem completes the study of normal forms of real-analytic Levi-flat hypersurfaces which are defined by the vanishing of real part of complex quasihomogeneous polynomials with isolated singularity. We also note that Theorem \ref{teorema1} generalizes the main results of \cite{burnsgong1999} and \cite{perez2014}, because the authors considered the same theorem for a \textit{generic Morse singularity} and \textit{Arnold singularities} of type $\mathcal{A}_k$, $\mathcal{D}_k$, $E_6$, $E_7$ and $E_8$,  which are given by complex quasihomogeneous polynomials with inner modality zero (see for instance \cite{arnold1974}).  
\par Topics about singular real-analytic Levi-flat hypersurfaces have been previously studied by several authors, see for instance \cite{bedford1977}, \cite{brunella2007}, \cite{perez2015}, \cite{perez2017}, \cite{lebl2013}, and normal forms of CR singular codimension two Levi-flat submanifolds was studied in \cite{gonglebl2015}. On the other hand, the study of normal forms of real-analytic hypersurfaces with Levi-form non-degenerate is given by the theory of Cartan \cite{cartan1932} and Chern-Moser \cite{chern}.
\par The second part of this paper is devoted to prove the existence of normal forms of singular real-analytic Levi-flat hypersurfaces which are preserved by a change of \textit{isochore} coordinates, that is, a change of coordinates that \textit{preserve volume}. Our main motivation are the Morse-type results for singularities of holomorphic functions given by J. Vey \cite{vey1977} and J-P Fran\c{c}oise \cite{francoise1978}. More precisely, Vey proved an isochore version of Lemma of Morse for germs of holomorphic functions at $0\in\C^n$, $n\geq 2$, and Fran\c{c}oise gave a new proof of the same result. A much more general statement was given by Garay \cite{garay2004}. In this same spirit, we propose here an analogous version of Vey's theorem for singular real-analytic Levi-flat hypersurfaces which are defined by the vanishing of the real part of a generic Morse function. We state the following result.
\begin{theoremno}\label{teorema2}
Let $M=\{F=0\}$ be a germ of an irreducible singular real-analytic Levi-flat hypersurface  at $0\in\mathbb{C}^n$, $n\geq 2$, such that
$$
F(z)=\mathcal{R}e (z_1^2+...+z_n^2)+H(z,\bar{z}),
$$ 
where $H(z,\bar{z})=O(|z|^{3})$ and $H(z,\bar{z})=\overline{H(\bar{z},z)}$. Then, there exists a germ of a volume-preserving biholomorphism $\phi:(\mathbb{C}^n,0)\to (\mathbb{C}^n,0)$ and a germ of an automorphism $\psi:(\mathbb{C},0)\to(\mathbb{C},0)$ such that
$$
\phi(M)=\{\mathcal{R}e(\psi(z_1^2+...+z_n^2))=0\}.$$
\end{theoremno}
\par The above theorem can be viewed as an \textit{isochore version} of Burns-Gong's theorem \cite{burnsgong1999}. On the other hand, in order to establish our next result we consider some definitions and notations that will be explained in the section 2:  for a germ of a singular real-analytic Levi-flat hypersurface $M$ with Levi foliation $\mathcal{L}$ and singular set $\sing(M)$, we will define the complexification $M_{\C}$ of $M$, which will be a germ of complex analytic subvariety contained $M$. The singular set of $M_{\C}$ will be denoted by $\sing(M_{\C})$. We will see that $M_{\C}$ is equipped with a germ of a singular codimension-one holomorphic foliation $\mathcal{L}_{\C}$, which will be the complexification of the Levi foliation $\mathcal{L}$. The singular set of $\mathcal{L}_{\C}$ will be denoted by $\sing(\mathcal{L}_{\C})$. 
\par Recently in  \cite{szawlowski2012}, A. Szawlowski presented a volume-preserving normal form for germs of holomorphic functions that are right-equivalent to the product of all coordinates. Motivated by \cite{szawlowski2012}, we will  prove an analogous version for singular real-analytic Levi-flat hypersurfaces.  
\begin{theoremno}\label{teorema3} 
Let $M=\{F=0\}$ be a germ of an irreducible singular real-analytic Levi-flat hypersurface at $0\in\mathbb{C}^n$, $n\geq 2$, such that
$F(z)= \mathcal{R}e(z_1\cdots z_n) + H(z,\bar{z})$,
where $H(z,\bar{z})=O(|z|^{n+1})$ and $H(z,\bar{z})=\overline{H(\bar{z},z)}$. Suppose that 
 $$\displaystyle\sing(M_{\C})=\bigcup_{\begin{tiny}
\begin{array}{c}
1\leq i<j\leq n\\1\leq k<\ell\leq n\end{array}\end{tiny}}V_{ijk\ell},$$
where $V_{ijk\ell}=\{(z,w)\in\C^{n}\times\C^{n}:z_i=z_j=w_k=w_{\ell}=0\}$ and $\sing(M_{\C})\subset \sing(\mathcal{L}_{\mathbb{C}})$.
Then, there exists a germ of codimension-one holomorphic foliation $\mathcal{F}_M$ tangent to $M$, with a non-constant holomorphic first integral $f(z)=z_1\cdots z_n+O(|z|^{n+1})$  such that $$M=\{\mathcal{R}e(f(z))=0\}.$$
\end{theoremno}
As consequence of above theorem and the main result of Szawlowski \cite{szawlowski2012} we have the following corollary.
\begin{corollary}\label{coro1}
Let $M$ be a germ of an irreducible singular real-analytic Levi-flat hypersurface as in Theorem \ref{teorema3}. If $f$ is right equivalent to the product of all coordinates, $f\sim_{R}z_1\cdots z_n$. Then there exists a germ of a volume-preserving biholomorphism $\Phi:(\mathbb{C}^n,0)\to (\mathbb{C}^n,0)$ and a germ of an automorphism $\Psi:(\mathbb{C},0)\to(\mathbb{C},0)$ such that
$$
\Phi(M)=\{\mathcal{R}e(\Psi(z_1\cdots z_n))=0\},$$
where $\Psi$ is uniquely determined by $f$ up to a sign. 
\end{corollary}
\par Let us recall that two germs of holomorphic functions $f$ and $g$ are right equivalent $f\sim_R g$, if there exist a germ of biholomorphism $\phi$ around the origin such that $f\circ \phi^{-1} = g$. We remark that the normal forms of Theorem \ref{teorema3} and Corollary \ref{coro1} are germs of real-analytic Levi-flat hypersurfaces whose singular set are of positive dimension. In general, the problem of finding normal forms of germs of real-analytic Levi-flat hypersurfaces with non-isolated singularities is very difficult and there are few results about it, see for instance \cite{perez2015}. 
\par To prove theorems 1, 2 and 3 we use the techniques of holomorphic foliations developed by D. Cerveau and A. Lins Neto in \cite{cerveauneto2011} and the first author in \cite{tesis}. These are fundamental in order to find normal forms of Levi-flat hypersurfaces. Specifically, we apply a result of Cerveau-Lins Neto that gives sufficient conditions for a real-analytic Levi-flat hypersurface to be defined by the zeros of the real part of a holomorphic function and a key Lemma that will be stated in section 5. 
\par This paper is organized as follows: in section 2, we recall some properties and known results about singular Levi-flat hypersurfaces. In section 3, we state some results about normal forms for a complex quasihomogeneous polynomial. Section $4$ is devoted to recall
the notions of weighted projective space and weighted blow-ups. In section $5$
we prove Theorem \ref{teorema1} and give an application of Theorem \ref{teorema1}. The section $6$ is dedicated to establish the isochore normal forms for holomorphic functions given by Vey and Szawlowski. In section $7$, we proved Theorem \ref{teorema2} and finally in section $8$, we proved Theorem \ref{teorema3} and Corollary \ref{coro1}.



\section{Singular Levi-flat hypersurfaces and holomorphic foliations}

The following notation will be used in this paper:
\begin{enumerate}
\item $\mathcal{O}_n$: the ring of germs of holomorphic functions at $0\in\C^n$.
\item $\mathcal{O}^*_n = \{ f\in\mathcal{O}_n\vert f(0)\neq 0\}$.
\item $\mathcal{M}_n=\{ f\in\mathcal{O}_n \vert f(0)=0\}$, the maximal ideal of $\mathcal{O}_n$.
\item $\mathcal{A}_n$: the ring of germs at $0\in\C^n$ of \textit{complex} valued real-analytic functions.
\item $\mathcal{A}_{n\mathbb{R}}$: the ring o germs of \textit{real} valued  functions. Note that $f\in\mathcal{A}_n\cap \mathcal{A}_{n\mathbb{R}}\iff f=\overline{f}$.
\item $j^k_0(f)$ is the $k$-jet at $0\in\C^n$ of $f\in\mathcal{O}_n$.
\item $\dif(\C^n,0)$: the group of germs of biholomorphisms $f:(\C^n,0)\rightarrow (\C^n,0)$ at $0\in\C^n$ with the operation of composition.

\end{enumerate}
\par Let $X$ be a compact connected complex manifold of complex dimension $n\geq 2$. A codimension-one singular \textit{holomorphic foliation} $\mathcal{F}$ on $X$ is given by a covering of $X$ by open subsets $\{U_j\}_{j\in J}$ and a collection of integrable holomorphic 1-forms $\omega_j$ on $U_j$, $\omega_j\wedge d\omega_j=0$, having zero set of complex codimension at least two such that, on each non-empty intersection $U_j\cap U_k$, we have 
\begin{equation}\label{colagem}
\omega_j=g_{jk}\omega_k,\,\,\,\,\,\,\,\text{with}\,\,\,\,\,\,\,\,g_{jk}\in\mathcal{O}^{*}(U_j\cap U_k).
\end{equation}
Let $\sing(\omega_j)=\{p\in U_j:\omega_j(p)=0\}$. Condition (\ref{colagem}) implies that $\sing(\mathcal{F}):=\displaystyle\bigcup_{j\in J}\sing(\omega_j)$ is a complex subvariety of complex codimension at least two in $X$.
\par Let $M$ be a germ of a real codimension-one irreducible real-analytic subvariety at $0\in\C^n$, $n\geq 2$. Without loss of generality we may assume that $M=\{F(z)=0\}$, where $F$ is  a germ of irreducible real-analytic function at $0\in\C^n$. We define the \textit{singular set} of $M$ as $$\sing(M)=\{ F(z)=0\}\cap\{dF(z)=0\}$$ and its \textit{regular part} is defined as $M^*=M\setminus \sing(M)$. Consider the distribution of complex hyperplanes $L$ on $M^*$ given by
$$
L_p:=ker(\partial F(p))\subset T_pM^*=ker(dF(p)),\text{ }p\in M^*.
$$
This distribution is called \textit{Levi distribution}. When $L$ is integrable, in the sense of Frobenius, then we say that $M$ is \textit{Levi-flat}. Since $M^*$ admits an integrable complex distribution, it is foliated locally by a real-analytic codimension-one foliation $\mathcal{L}$ on $M^{*}$, the \textit{Levi foliation}. Each leaf of $\mathcal{L}$ is a codimension-one holomorphic submanifold immersed in $M^*$.
\par The distribution $L$ can be defined by the real-analytic 1-form $\eta = i(\partial F-\bar{\partial} F)$, the \textit{Levi form} of $F$. The integrability condition is equivalent to
$$
(\partial F-\bpartial F)\wedge \partial\bpartial F\vert_{M^*}=0
$$
which using the fact that $\partial F +\bpartial F = dF$, is equivalent to
$$
\partial F (p)\wedge \bpartial F(p) \wedge \partial \bpartial F (p) = 0\,\,\,\,\,\,\,\,\,\,\forall\,\, p\in M^{*}.
$$
We refer to the book \cite{salah1999} for the basic language and background about Levi-flat hypersurfaces. 
\par Suppose that $M$ is Levi-flat as above. If $\sing(M)=\emptyset$, then we say that $M$ is \textit{smooth}. In this case, according to Cartan \cite{cartan1932}, around the origin of $\mathbb{C}^n$ one may find suitable coordinates $(z_1,...,z_n)$ of $\C^n$ such that the germ of $M$ at $0\in\mathbb{C}^n$ is given by $$\{\mathcal{R}e(z_n)=0\}.$$ This is called the \textit{local normal form} for a smooth real-analytic Levi-flat hypersurface $M$ at $0\in\mathbb{C}^n$. 
\par In order to build singular real-analytic Levi-flat hypersurfaces which are irreducible, we consider the following lemma from \cite{cerveauneto2011}.
\begin{lemma}\label{lema1} Let $f\in\mathcal{O}_n$, $f\neq 0$, $f(0)=0$ which is not a power in $\mathcal{O}_n$. Then $\mathcal{I}m(f)$ and $\mathcal{R}e(f)$ are irreducible in $\mathcal{A}_{n\mathbb{R}}$. 
\end{lemma} 

 Before proving our results, let us describe some known results and examples. 
 \begin{example} 
Let $f\in\mathcal{O}_n$ be a germ of non-constant holomorphic function with $f(0)=0$. Then the  set $M=\{\mathcal{R}e(f)=0\}$ is Levi-flat and its singular set is given by $\crit(f)\cap M$, where $\crit(f)$ is the set of critical points of $f$. The leaves of the Levi foliation $\mathcal{L}$ on $M$ are the imaginary levels of $f$.
\end{example}
\begin{example}
In $\C^n$, $n\geq 2$, let $M$ be given as the set of zeros of
$$
F(z_1,z_2,...,z_n)=z_1\overline{z}_1-z_2\overline{z}_2.
$$
Then $M$ is Levi-flat and its singular set biholomorphic to $\C^{n-2}$. This real-analytic hypersurface is called \emph{quadratic complex cone}. The leaves of the Levi foliation $\mathcal{L}$ on $M$ are the hyperplanes 
$$L_c=\{(z_1,z_2,...,z_n)\in\mathbb{C}^n:z_1-c\cdot z_2=0\}\,\,\,\,\,\text{where}\,\,\,\,\,\,c\in\mathbb{R}.$$
\end{example}
\begin{example}
Let $M$ be a germ of real-analytic hypersurface at $0\in\C^n$ given by $\{ F=0\}$, where
$$
F(z_1,...,z_n) = \mathcal{R}e(z_1^2+...+z_n^2)+H(z,\bar{z}),\,\,\,\,\,\text{and}\,\,\,\,\,H(z,\bar{z})=O(|z|^3).
$$
If $M$ is Levi-flat, then, according to \cite{burnsgong1999}, there exists a holomorphic coordinate system such that $M=\{\mathcal{R}e(x_1^2+...+x_n^2)=0\}$. We remark that this result was generalized in \cite{perez2011}, where the first author considered the real part of a complex homogeneous polynomial of degree $k\geq 2$ with an isolated singularity. 
\end{example}

\begin{example}
We considere the famous $A_{k},D_{k},E_{k}$ singularities or simple singularities of Arnold  \cite{arnold1972}, \cite{arnold1974}:
$$\begin{tabular}{|c|l|r|}\hline
Type & Normal form & Conditions \\
\hline
$A_{k}$ & $z_{1}^{2}+z_{2}^{k+1}+\ldots +z^{2}_{n},$ & $k\geq 1$\\
$D_{k}$ & $z_{1}^{2}z_{2}+z_{2}^{k-1}+z_{3}^{2}+\ldots+z_{n}^{2},$ & $k\geq 4$\\
$E_{6}$ & $z_{1}^{4}+z_{2}^{3}+z_{3}^{2}+\ldots+z_{n}^{2}$ &  \\
$E_{7}$ & $z_{1}^{3}z_{2}+z_{2}^{3}+z_{3}^{2}+\ldots+z_{n}^{2}$ &  \\
$E_{8}$ & $z_{1}^{5}+z_{2}^{3}+z_{3}^{2}+\ldots+z_{n}^{2}$ &  \\
\hline

\end{tabular}$$
Let $M$ be a germ of singular real-analytic Levi-flat hypersurface at $0\in\mathbb{C}^2$ defined by $\{F=0\}$, where 
$$F(z)=\mathcal{R}e(Q(z))+H(z,\bar{z}),$$
where $Q$ is a complex quasihomogeneous polynomial of $A_{k}$, $D_{k},$ or $E_{k}$ type of quasihomogeneous degree $d$. Then in \cite{perez2014} it has been proved that there exists a holomorphic coordinate system such that 
$$M=\{\mathcal{R}e(Q(z))=0\}.$$ 
We remark that, in this case, the elements $e_{1},\ldots,e_{s}$ of the monomial basis of the local algebra of $Q$ of quasihomogeneous degree strictly greater than $d$ are zero, because the inner modality of the $A_{k},D_{k},E_{k}$ singularities are zero. 
 
\end{example}

\subsection{Complexification of singular Levi-flat hypersurfaces}\label{complexification}

Let $M$ be a germ of a singular real-analytic Levi-flat hypersurface at $0\in\mathbb{C}^n$ defined by the set of zeros of $F\in\mathcal{A}_{n\mathbb{R}}$. Let $\sing(M)$, $M^*$ and $\mathcal{L}$ be the singular set, the regular part and the Levi foliation on $M^{*}$ respectively. 
\par We write the Taylor series of $F$ around $0\in\C^n$ as
$$
F(z) = \sum_{\mu,\nu} F_{\mu\nu}z^\mu\bar{z}^\nu,
$$
where $F_{\mu\nu}\in\mathbb{C}$, $\mu=(\mu_1,...,\mu_n)$, $\nu=(\nu_1,...,\nu_n)$, $z^\mu = z_1^{\mu_1}\cdots z_n^{\mu_n}$ and $\bar{z}^\nu = \bar{z}_1^{\nu_1}\cdots \bar{z}_n^{\nu_n}$. Since $F\in\mathcal{A}_{n\mathbb{R}}$, the coefficients verify $\overline{F}_{\mu \nu} = F_{\nu \mu}$. We define the \emph{complexification} $F_\C\in\mathcal{O}_{2n}$ of $F$ as the function defined by the power series $$F_\C(z,w) = \sum_{\mu,\nu} F_{\mu\nu}z^{\mu}w^{\nu}.$$
If the power series for $F$ converges in a polydisc $D^n_r = \{ z\in\C^n : \vert z_j\vert \leq r\}$ then the power series of the complexification $F_\C$ of $F$ is convergent in the polydisc $D^{2n}_r$ and therefore is holomorphic at $0\in\C^{2n}$. Moreover, $$F(z)=F_\C(z,\bar{z})\,\,\,\,\,\,\,\forall z\in D^n_r.$$ This complexification does not depend on choice of coordinate system, see for instance  \cite{cerveauneto2011}.
\par As seen before, the Levi 1-form is given by $\eta=i(\partial F - \bar{\partial} F)$. Its \emph{complexification} is the germ of holomorphic 1-form
$$
\eta_\C = i\sum_{j=1}^n \left( \dfrac{\partial F_\C}{\partial z_j}dz_j - \dfrac{\partial F_\C}{\partial w_j}dw_j \right) =i\sum_{\mu,\nu} (F_{\mu\nu}w^\nu d(z^\mu) - F_{\mu\nu}z^{\mu}d(w^\nu)).
$$
The complexification of $M$ is defined as $M_\C = \{F_\C=0\}$. As before, $M_\C$ does not depend on choice of coordinate system. The \emph{regular part} of $M_\C$ is 
$$M^*_\C = M_\C\setminus \{dF_\C = 0\}$$ and the singular part of $M_\C$ is 
$$\sing(M_{\C})=M_\C\cap \{dF_\C = 0\}.$$ Since $\eta$ is integrable on $M^{*}$, then also $\eta_\C\vert _{M^*_{\C}}$ is integrable and defines a codimension-one holomorphic foliation on $M^*_\C$, which will be denoted by $\mathcal{L}_\C$. Such foliation is called \emph{complexification} of $\mathcal{L}$.
\begin{remark}\label{remark_fol}
We can write $\eta_\C = i(\alpha - \beta)$, where 
$$\alpha:=\sum^{n}_{j=1} \dfrac{\partial F_\C}{\partial z_j}dz_j\,\,\,\,\,\text{and}\,\,\,\,\,\,\,\beta:=\sum^{n}_{j=1}\dfrac{\partial F_\C}{\partial w_j}dw_j.$$
Note that $dF_\C=\alpha+\beta$, then 
$$
\eta_\C\vert_{M^*_\C}=(\eta_\C+idF_\C)\vert_{M^*_\C}=2i\alpha\vert_{M^*_\C}.
$$
Analogously $$
\eta_\C\vert_{M^*_\C}=(\eta_\C-idF_\C)\vert_{M^*_\C}=-2i\beta\vert_{M^*_\C}.$$
In particular, $\alpha\vert_{M^*_\C}$ and $\beta\vert_{M^*_\C}$ define $\mathcal{L}_{\C}$ on $M^*_\C$ and $\sing(\mathcal{L}_{\C})= \sing(\eta_\C\vert_{M^*_\C})$.
\end{remark}
\begin{definition}
Let $M=\{F=0\}$ be a germ at $0\in\C^n$ of a real-analytic Levi-flat hypersurface and $M_\C$ its complexification. We define the \emph{algebraic dimension} of $\sing(M)$ as the complex dimension of $\sing(M_\C)$.
\end{definition}

Let $W=M^*_\C\setminus \sing(\eta_\C\vert_{M^*_\C})$ and let $L_p$ be the leaf of $\mathcal{L}_\C$ through $p\in W$. We have the following lemma from \cite{cerveauneto2011}.

\begin{lemma}\label{lema2} For any $p\in W$, the leaf $L_p$ is closed (with the induced topology) in $M^*_\C$.
\end{lemma}

The following theorem, due to D. Cerveau and A. Lins Neto \cite{cerveauneto2011} is the key ingredient for finding normal forms of singular Levi-flat hypersurfaces.  

\begin{theorem}\label{cerveaulinsneto1} Let $M=\{F=0\}$ be a germ  of an irreducible real-analytic Levi-flat hypersurface at $0\in\mathbb{C}^n$, $n\geq 2$, with Levi 1-form $\eta$. Assume that the algebraic dimension of $\sing(M)$ is at most $2n-4$. Then there exists a unique germ at $0\in\C^n$ of codimension-one holomorphic foliation $\mathcal{F}_M$ tangent to $M$, if one of the following conditions is fulfilled:
\begin{itemize}
\item[(a)] $n\geq 3$ and $\cod_{M^*_\C} (\sing (\eta_\C\vert_{M^*_\C}))\geq 3$.
\item[(b)] $n\geq 2$, $\cod_{M^*_\C} (\sing (\eta_\C\vert_{M^*_\C}))\geq 2$ and $\mathcal{L}_\C$ has a non-constant holomorphic first integral.
\end{itemize}
Moreover, in both cases the foliation $\mathcal{F}_M$ has a non-constant holomorphic first integral $f$ such that $M=\{\mathcal{R}e(f)=0\}$.
\end{theorem}
\par We recall that germ of holomorphic function $h$ is called a \textit{holomorphic first integral} for a germ of codimension-one holomorphic foliation $\mathcal{F}$ if its zeros set is contained in $\sing(\mathcal{F})$ and its level hypersurfaces contain the leaves of $\mathcal{F}$.
\section{Normal forms for a quasihomogeneous polynomial}
The \textit{local algebra} of $f\in\mathcal{O}_n$ is defined as
$$
A_f = \frac{\mathcal{O}_n}{ \left\langle \frac{\partial f}{\partial z_1},...,\frac{\partial f}{\partial z_n}\right\rangle}.
$$
The number $\mu(f,0)=dim_{\C}(A_f)$ is the Milnor number of $f$ at $0\in\C^n$. This number is finite if and only if $f$ has an isolated singularity at the origin. With these definitions, Morse  lemma may be stated as follows: if $0\in\C^n$ is an isolated singularity of $f\in\mathcal{O}_n$ with $\mu(f,0)=1$, then $f$ is right equivalent to is second jet $j_0^2(f)$. The Morse lemma has the following generalization, and the proof can be found in \cite{arnold1985}.

\begin{theorem}
 If $f\in\mathcal{M}_n$ has an isolated singularity at $0\in\C^n$ with Milnor number $\mu$, then $f$ is right equivalent to $j_0^{\mu+1} (f)$.
\end{theorem} 
\begin{definition}\label{definicaoquasihomogeneo}
A germ of function $f\in\mathcal{O}_n$ is \textit{quasihomogeneous} with weights $w_1, ...,w_n\in\mathbb{Z}^*_+$ if, for each $\lambda\in\C^*$,
$$
f(\lambda^{w_1} z_1,...,\lambda^{w_n} z_n)=\lambda^d f(z_1,...,z_n).
$$
The number $d$ is the \textit{quasihomogeneous degree} of it.
\end{definition}
The previous definition is equivalent to the following: $f(z)$ is quasihomogeneous of \textit{type} $(w_1,...,w_n)$ if it can be expressed as a linear combination of monomials $z_1^{i_1}z_2^{i_2}\cdots z_n^{i_n}$ for which the equality
$$
i_1w_1+...+i_nw_n=d
$$
holds. The number $d$ is the quasihomogeneous degree defined above. 

\begin{definition}  
 The Newton support of germ $f=\sum a_{i_1\ldots i_n}x^{i_1}\ldots x^{i_n}$ is defined as 
$$supp(f)=\{(i_1,\ldots,i_n): a_{i_1\ldots i_n}\neq 0\}.$$
\end{definition}

In the above situation, if $f=\sum a_{I}x^{I}$, $I=(i_{1},\ldots,i_{n})$, $x^{I}=x_{1}^{i_{1}}\ldots x^{i_{n}}$, then $$supp(f)\subset\Gamma=\{I:w_{1}i_{1}+\ldots+w_{n}i_{n}=d\}.$$ The set $\Gamma$ is called the diagonal. One can define the quasihomogeneous filtration of the ring $\mathcal{O}_{n}$. It consists of the decreasing family of ideals $\mathcal{A}_{d}\subset\mathcal{O}_{n}$, $\mathcal{A}_{d'}\subset\mathcal{A}_{d}$ for $d<d'$. Here $\mathcal{A}_{d}=\{Q:$ degrees of monomials from $supp(Q)$ are $deg(Q)\geq d\}$; (the degree is quasihomogeneous). When $i_{1}=\ldots=i_{n}=1$, this filtration coincides with the usual filtration by the usual degree.

\begin{definition} A function $f$ is \emph{semiquasihomogeneous} if $f=Q+F'$, where $Q$ is quasihomogeneous of quasihomogeneous degree $d$ and $\mu(Q,0)<\infty$, and $F'\in\mathcal{A}_{d'}$, $d'>d$.
\end{definition}

From \cite{arnold1974} we have the following result of V.I. Arnold.

\begin{theorem}\label{lemaarnold}
Let $f=Q+F'$ be a semiquasihomogeneous function. Then $f$ is right-equivalent to a function $\displaystyle Q(z)+\sum_j c_je_j(z)$ where $e_1,...,e_j$ are elements of the monomial basis of the local algebra $A_Q$ of quasihomogeneous degree strictly greater than $d$
 and $c_j\in \C$.
\end{theorem}

\begin{example}
 Let $f=Q+F'$, where $Q(x,y)=x^{2}y+y^{k}$, then $f$ is right equivalent to $Q$. Indeed, the basis of the local algebra $$A_Q=\mathcal{O}_{2}/\langle xy,x^{2}+ky^{k-1}\rangle$$  is $1,x,y,y^{2},\ldots,y^{k-1}$. Here $\mu(Q,0)=k+1$.
\end{example}
In the proof of Theorem \ref{teorema1}, we will used the following Lemma of Saito \cite{saito1971}.
\begin{lemma}\label{lemaqh}
If $f\in\mathcal{M}_2$ is a complex quasihomogeneous polynomial, then $f$ factors itself uniquely as
$$
f(z_1,z_2)=\mu z_1^mz_2^n\prod_{\ell=1}^k (z_2^p-\lambda_{\ell} z_1^q),
$$
where $m,n,p,q\in\mathbb{Z}_+^*$,  $\mu,\lambda_{\ell}\in\C^*$ for each $\ell=1,...,k$, and $\gcd(p,q)=1$.
\end{lemma}

\section{Weighted projective varieties and weighted blow-ups}

In this section we present an overview of weighted projective spaces and weighted blow-ups. We refer to \cite{dolgachev1982} and \cite {kollar2007} for a more extensive presentation of the subject.

\par Let $\sigma:=(a_{0},\ldots,a_{n})$ be positive integers. The group $\mathbb{C}^{*}$ acts on
$\mathbb{C}^{n+1}\backslash\{0\}$ by
$$\lambda\cdot(x_{0},\ldots,x_{n})=(\lambda^{a_{0}}x_{0},\ldots,\lambda^{a_{n}}x_{n}).$$ The quotient space under this action is the weighted projective space of
type $\sigma$, $\mathbb{P}(a_{0},\ldots,a_{n}):=\mathbb{P}_{\sigma}$. In case $a_{i}>1$ for some $i$, $\mathbb{P}_{\sigma}$ is a compact algebraic variety with cyclic
quotient singularities.
\par Let $[x_{0}:\ldots:x_{n}]$ be the homogeneous coordinates on $\mathbb{P}(a_{0},\ldots,a_{n})$. The affine piece $x_{i}\neq 0$ is isomorphic
to $\mathbb{C}^{n}/\mathbb{Z}_{a_{i}}$, here $\mathbb{Z}_{a_{i}}$ denote the quotient group modulo $a_{i}$. 
Let $\epsilon$ be an $a_{i}^{th}$-primitive root of unity. The group acts by
$$z_{j}\longmapsto\epsilon^{a_{j}}z_{j}$$
for all $j\neq i$, on the coordinates $(z_{0},\ldots,\hat{z_{i}},\ldots,z_{n})$ of $\mathbb{C}^{n}$; here $z_{j}$ is thought of
 as $x_{j}/x_{i}^{1/a_{i}}$. Compare this to the case of $\mathbb{P}^{n}$ where the affine coordinates on $x_{i}\neq 0$ are $z_{j}=x_{j}/x_{i}$.

\begin{definition}
 $\mathbb{P}(a_{0},\ldots,a_{n})$ is well-formed if  for each i
$$\gcd(a_{0},\ldots,\hat{a}_{i},\ldots,a_{n})=1.$$
\end{definition}
\noindent We have a natural orbifold map $\phi_{\sigma}:\mathbb{P}^{n}\rightarrow\mathbb{P}_{\sigma}$ defined by
\begin{align}\label{quotient-map}
[x_{0}:\ldots:x_{n}]\mapsto[x_{0} ^{a_{0}}:\ldots:x_{n} ^{a_{n}}]_{\sigma}
\end{align}

\begin{definition}
Let $X$ be a closed subvariety of a weighted projective space $\mathbb{P}_{\sigma}$, and let
$\rho:\mathbb{C}^{n+1}\backslash\{0\}\rightarrow\mathbb{P}_{\sigma}$ be the canonical projection.
The punctured affine cone $C^{*}_{X}$ over $X$ is given by $C^{*}_{X}=\rho^{-1}(X)$, and the affine cone $C_{X}$ over $X$ is
the completion of $C^{*}_{X}$ in $\mathbb{C}^{n+1}$.
\end{definition}
\par Observe that $\mathbb{C}^{*}$ acts on $C^{*}_{X}$ giving $X=C^{*}_{X}/\mathbb{C}^{*}$. Note that we have the following fact. 
\begin{lemma}
$C^{*}_{X}$ has no isolated singularities.
\end{lemma}
\begin{definition}
We say that $X$ in $\mathbb{P}_{\sigma}$ is quasi-smooth of dimension $m$ if its affine cone $C_{X}$ is smooth of dimension $m+1$ outside its vertex $0\in\mathbb{C}^{n+1}$.
\end{definition}
\par When $X\subset\mathbb{P}_{\sigma}$ is quasi-smooth the singularities of $X$ are given by the $\mathbb{C}^{*}$-action and hence are
cyclic quotient singularities. Notice that this definition is not equivalent to the smoothness
of the inverse image $\phi_{\sigma}^{-1}(X)$ under the quotient map given in (\ref{quotient-map}).
\par Another important fact (cf. \cite{dolgachev1982}, Theorem 3.1.6) is that a quasi-smooth subvariety $X$ of $\mathbb{P}_{\sigma}$ is a
$V$-variety, that is, a complex space which is locally
isomorphic to the quotient of a complex manifold by a finite group of holomorphic automorphisms.
\par Now, let $X=\mathbb{C}^n/\mathbb{Z}_{m}(a_{1},\ldots,a_{n})$ be a cyclic quotient singularity. That is, $X$ is the quotient variety $\mathbb{C}^{n}/\tau$,
where $\tau$ is given by
$$x_{i}\longmapsto\epsilon^{a_{i}}x_{i}$$
for all i, where $\epsilon$ is a $m^{th}$-primitive root of unity.  
\subsection{Weighted blow-ups}
In this part we will construct the blow-up of $X$. First, we describe $X$ using the theory of toric varieties (cf. \cite{fulton1993}). Let
\begin{align*}
e_{1}=(1,0,\ldots,0),\ldots,e_{n}=(0,\ldots,0,1)\,\, \text{and}\,\, e=\frac{1}{m}(a_{1},\ldots,a_{n}).
\end{align*}
Then $X=\mathbb{C}^n/\mathbb{Z}_{m}(a_{1},\ldots,a_{n})$ is the toric variety corresponding to the lattice $N=\mathbb{Z}e_{1}+\ldots+\mathbb{Z}e_{n}
+\mathbb{Z}e$ and the cone $C=\mathbb{R}_{\geq 0}e_{1}+\ldots +\mathbb{R}_{\geq 0}e_{n}$. Denote by $\bigtriangleup$ the fan associated to $X$ consisting of all the faces of $C$.
\par Take $\nu=\frac{1}{m}(a_{1},\ldots,a_{n})\in N$ with $a_{1},\ldots,a_{n}>0$ and assume that $e_{1},\ldots,e_{n}$ and $\nu$
generate the lattice $N$. Such $\nu\in N$ will be called a weight. We can construct the weighted blow-up
$$E:\tilde{X}\rightarrow X=\mathbb{C}^n/\mathbb{Z}_{m}(a_{1},\ldots,a_{n})$$ with weight $\nu$ as follows: we divide the cone $C$ by adding the 1-dimensional cone $\mathbb{R}_{\geq 0}\nu$, that is, we divide $C$ into $n$ cones
\begin{align*}
C_{i}=\mathbb{R}_{\geq 0}e_{1}+\ldots+\overset{i-th}{\overbrace{\mathbb{R}_{\geq 0}\nu}}+\ldots+\mathbb{R}_{\geq 0}e_{n}\,\,\,\,\,\,\,(i=1,\ldots,n).
\end{align*}
Let $\bigtriangleup'$ be the fan consisting of all the faces of $C_{1},\ldots,C_{n}$. Then $\tilde{X}$ is the toric variety corresponding to $N$
 and $\bigtriangleup'$, while $E$ is the morphism induced from the natural map of fans $(N,\bigtriangleup')\rightarrow(N,\bigtriangleup)$.
\par The variety $\tilde{X}$ is covered by $n$ affine open sets $\tilde{U}_{1},\ldots,\tilde{U}_{n}$ which correspond to the cones $C_{1},\ldots,C_{n}$ respectively.
These affine open sets and $E$ are described as follows:
\begin{equation}
\tilde{U_{i}}=\mathbb{C}^n/\mathbb{Z}_{a_{i}}(-a_{1},\ldots,\overset{i-th}{\overbrace{m}},\ldots,-a_{n})
\end{equation}
\begin{equation}
E|_{\tilde{U_{i}}}:\tilde{U_{i}}\ni(y_{1},\ldots,y_{n})\longmapsto(y_{1}y_{i}^{a_{1}/m},\ldots,\overset{i-th}{\overbrace{y_{i} ^{a_{i}/m}}},\ldots,y_{n}y_{i} ^{a_{n}/m})\in X.
\end{equation}
\par The exceptional divisor $D$ of $E$ is isomorphic to the weighted projective space $\mathbb{P}(a_{1},\ldots,a_{n})$ and
$D\cap\tilde{U_{i}}=\{y_{i}=0\}/\mathbb{Z}_{a_{i}}$.

\section{First integral for the Levi foliation and the proof of Theorem \ref{teorema1}}
In this section, we give sufficient conditions (dynamical criteria) to find a non-constant holomorphic first integral for the complexification of the Levi foliation $\mathcal{L}_{\mathbb{C}}$ on $M_{\mathbb{C}}$ and then we prove Theorem \ref{teorema1}.
\par Let $\pi$ be a weighted blow-up  on  $M_{\mathbb{C}}$ with exceptional divisor $E$. Denote by $\tilde{M}_\C$ the strict transform of $M_\C$ by $\pi$ and by $\tilde{\mathcal{F}}=\pi^*(\mathcal{L}_\C)$ the induced foliation on $\tilde{M}_\C$. Suppose that $\tilde{M}_\C$ is a smooth variety and consider $\tilde{C}=\tilde{M}_\C\cap E$. Assume that $\tilde{C}$ is invariant by $\tilde{\mathcal{F}}$; i.e., it is a union of leaves and singularities of $\tilde{\mathcal{F}}$. 
\par Let $S:=\tilde{C}\setminus\sing(\tilde{\mathcal{F}})$. Then $S$ is a smooth leaf of $\tilde{\mathcal{F}}$. Take a point $p_0$ in $S$ and a transverse section $\Sigma$ passing through $p_0$. Let $G\subset\dif(\Sigma, p_0)$ be the holonomy group of the leaf $S$; since $dim(\Sigma)=1$, we assume that $G\subset \dif(\Sigma,0)$. In this context, we have the following result of Fern\'andez-P\'erez \cite{perez2014}.
\begin{lemma}\label{lemaarturo} Assume the following:
\begin{itemize}
\item[(a)] For any $p\in S\setminus \sing(\tilde{\mathcal{F}})$, the leaf $L_p$ of $\tilde{\mathcal{F}}$ through $p$ is closed in $S$.
\item[(b)] $g'(0)$ is a primitive root of unity, for all $g\in G$, $g\neq id$.
\end{itemize}
Then $\mathcal{L}_\C$ has a non-constant holomorphic first integral.
\end{lemma}
 \par To continue, we use the above lemma to prove the following proposition.
 \begin{proposition}\label{proposition_integral}
 Let $M$ be a germ of an irreducible singular real-analytic Levi-flat hypersurface at $0\in\mathbb{C}^2$ satisfying the hypotheses of Theorem \ref{teorema1}.
Then we have the following:
\begin{enumerate}
\item[(a)] the algebraic dimension of $\sing(M)$ is $0$;
\item[(b)] $\cod_{M^*_\C} (\sing (\mathcal{L}_{\C}))=2$;
\item[(c)] $\mathcal{L}_\C$ has a non-constant holomorphic first integral. 
\end{enumerate}
 \end{proposition} 
 \begin{proof}
Let $M$ be as in Theorem \ref{teorema1}. Then $M$ is given by $M=\{F=0\}$, where 
$$F(z)=\mathcal{R}e(Q(z))+H(z,\bar{z}),$$
$Q$ is a complex quasihomogeneous polynomial of quasihomogeneous degree $d$ of type $(a,b)$ with an isolated singularity at $0\in\mathbb{C}^{2}$ and $H$ is a germ of real-analytic function at $0\in\mathbb{C}^2$ of order strictly greater than $d$.
It follows from Lemma \ref{lemaqh} that $Q$ can be written as  
\begin{equation}\label{qh}
Q(x,y)=\mu x^m y^n\prod_{\ell=1}^k (y^p-\lambda_{\ell} x^q),
\end{equation}
where $m,n,p,q\in\mathbb{Z}_+^*$,  $\mu,\lambda_{\ell}\in\C^*$ for each $\ell=1,...,k$, and $\gcd(p,q)=1$. Since $Q$ has an isolated singularity at $0\in\C^2$, then we necessarily 
that both $m$ and $n$ are either 0 or 1. 
\par On the other hand, since $Q$ has weights $(a,b)$ with $\gcd(a,b)=1$ we have each polynomial $(y^p-\lambda_{\ell}x^q)$ has also weights $(a,b)$, which implies that $aq=bp$. Since $p,q$ are relatively prime, we get $a=p$ and $b=q$. 
\par For simplicity, using (\ref{qh}), we write
$$
Q(x,y)=\mu x^m y^n\prod_{\ell=1}^k Q_\ell(x,y), 
$$
where $Q_\ell(x,y)=(y^p-\lambda_{\ell}x^q)$. Without loss of generality, we can assume that $Q$ has real coefficients. Then the complexification $F_\C$ of $F$ is given by
$$
F_\C(x,y,z,w)=\dfrac{1}{2}Q(x,y)+\dfrac{1}{2}Q(z,w)+H_\C(x,y,z,w).
$$
Since $Q$ has an isolated singularity at $0\in\C^2$, we get $M_\C=\{F_\C=0\}\subset (\C^4,0)$ has an isolated singularity at $0\in\C^4$ and so the algebraic dimension of $\sing(M)$ is zero. Hence item $(a)$ is proved. Consider the algebraic subvariety contained in  $\Proj(a,b,a,b)$
$$
V_{M_\C}=\{ Q(Z_0,Z_1)+Q(Z_2,Z_3)=0\},
$$
where $[Z_0:Z_1:Z_2:Z_3]\in\Proj(a,b,a,b)$. It is not difficult to see that  $\sing(M_\C)=\sing(V_{M_\C})$. Note that $V_{M_\C}$ can be considered as $V$-variety
$$
V_{M_\C}\subset Z \simeq \C^4/\mathbb{Z} (a,b,a,b).
$$
Now we consider the weighted blow-up $E:\tilde{Z}\rightarrow Z$, with weight $\delta=(a,b,a,b)$. Let $\tilde{M}_\C$ be the strict transform of $M_\C$ by $E$ and $D\simeq \Proj_\delta$ the exceptional divisor, with coordinates $(Z_0,Z_1,Z_2,Z_3)\in\C^4\setminus\{0\}$. The intersection of $\tilde{M_\C }$ with $\Proj_\delta$ is
$$
\tilde{C}:=\tilde{M}_\C \cap\Proj_\delta=\{ Q(Z_0,Z_1) + Q(Z_2,Z_3) = 0\}.
$$
It follows from Remark \ref{remark_fol} that $\mathcal{L}_\C$ can be defined by $\alpha\vert_{M^*_\C}=0$, where
\begin{eqnarray}\label{forma_1}
\alpha &=&Q(x,y)\left[\left(\dfrac{m}{x}-qx^{q-1}\displaystyle\sum_{\ell=1}^k \dfrac{\lambda_{\ell}}{Q_\ell(x,y)}\right)dx\right.+\left.\left( \dfrac{n}{y}+py^{p-1}\displaystyle\sum_{\ell=1}^k\dfrac{1}{Q_\ell(x,y)} \right)dy \right]+\theta
\end{eqnarray}
and  $\theta=2\left(\frac{\partial{H}_{\mathbb{C}}}{\partial{x}}dx+\frac{\partial{H}_{\C}}{\partial{y}}dy\right)$ is a holomorphic 1-form with order strictly greater than $d$.  It follows from (\ref{forma_1}) that $\sing(\mathcal{L}_\C)$ has codimension two proving item $(b)$. 
The rest of the proof is devoted to the proof of item $(c)$. Note that the leaves of $\mathcal{L}_{\C}$ are closed in $M^{*}_{\C}\setminus\sing(\mathcal{L}_{\C})$ by Lemma \ref{lema2}. To apply Lemma \ref{lemaarturo} we need calculate the holonomy group associated to $\mathcal{L}_{\C}$.                                                       
\par For each $i=1,2,3,4$, we have the affine open sets
$$
\tilde{U}_i=\C^4/ \mathbb{Z}_{a_i} (-a,...,\underset{i-th}{\underbrace{1}},...,-b).
$$
We work in $\tilde{U}_3$ with coordinates $(x_1,y_1,z_1,w_1)$. In this open subset, the blow-up $E$ has the following expression
$$
E(x_1,y_1,z_1,w_1)=(x_1z_1^a,y_1z_1^b,z_1^a,w_1z_1^b),
$$
with $D\cap \tilde{U}_3=\left\{ z_1=0\right\}/ \mathbb{Z}_{a}$. In this chart, the pull-back of $\alpha$ by $E$ is given by
\begin{eqnarray*}
E^*\alpha = z_1^{pm+qn+kpq-1} \alpha_1,
\end{eqnarray*}
where 
\begin{eqnarray}\label{alpha_teo1}
\alpha_1&=&Q(x_1,y_1)\left[ \left( \dfrac{mz_1}{x_1}-qx_1^{q-1}z_1\displaystyle\sum^k_{\ell=1}\dfrac{\lambda_{\ell}}{Q_\ell(x_1,y_1)} \right)dx_1+\right.  \\
&  &\left. \left( \dfrac{nz_1}{y_1}+py_1^{p-1}z_1\displaystyle\sum_{\ell=1}^k\dfrac{1}{Q_\ell(x_1,y_1)} \right)dy_1 +\right. \nonumber\\ 
& & \left. \left( pm+qn-pqx_1^q\displaystyle\sum_{\ell=1}^k\dfrac{\lambda_{\ell}}{Q_\ell(x_1,y_1)}+pqy_1^p\displaystyle\sum_{\ell=1}^k\dfrac{1}{Q_\ell(x_1,y_1)} \right)dz_1\right ]+z_1\theta_1\nonumber
\end{eqnarray}
and $\theta_1=E^*\theta/z_1^{{pm+qn+kpq}}$. The pull-back foliation $\tilde{\mathcal{L}}_\C$ is defined by $\alpha_1|_{\tilde{M}^{*}_{\C}}=0$. The intersection of $\tilde{C}$ with the open subset $\tilde{U}_3$ is
$$
\tilde{C}\cap \tilde{U}_3=\{ z_1=Q(x_1,y_1)+Q(1,w_1)=0\}/\mathbb{Z}_{a},
$$
which implies that $\tilde{C}$ is invariant by $\tilde{\mathcal{L}}_\C$ by (\ref{alpha_teo1}), and
$$
\sing (\tilde{\mathcal{L}}_\C)\cap \tilde{U}_3=\{ z_1=Q(x_1,y_1)=Q(1,w_1)=0\}/\mathbb{Z}_a.
$$
In the chart $\tilde{U}_4$, with coordinates $(x_2,y_2,z_2,w_2)$, the blow-up is
$$
E(x_2,y_2,z_2,w_2)= (x_2w_2^a,y_2w_2^b,z_2w_2^a,w_2^b)
$$
and $D\cap \tilde{U}_4=\{ w_2=0\}/\mathbb{Z}_b$. In this chart, the pull-back of $\alpha$ is
$$
\begin{array}{llr}
E^*\alpha = w_2^{pm+qn+kpq-1}\alpha_2,
\end{array}
$$
where 
\begin{eqnarray}\label{equa_2}
\alpha_2&=&Q(x_2,y_2)\left[
\left( \dfrac{mw_2}{x_2}-qx_2^{q-1}w_2\sum^k_{\ell=1}\dfrac{\lambda_{\ell}}{Q_\ell(x_2,y_2)} \right)dx_2+\right. \\
&  & \left. \left( \dfrac{nw_2}{y_2}+py_2^{p-1}w_2\sum_{\ell=1}^k\dfrac{1}{Q_\ell(x_2,y_2)} \right)dy_2+\right.\nonumber\\
& & \left. \left( pm+qn-pqx_2^q\sum_{\ell=1}^k\dfrac{\lambda_{\ell}}{Q_\ell(x_2,y_2)}+pqy_2^p\sum_{\ell=1}^k\dfrac{1}{Q_\ell(x_2,y_2)} \right)dw_2\right] +w_2\theta_2 \nonumber
\end{eqnarray}
and $\theta_2=E^*\theta/w_2^{{pm+qn+kpq}}$. The pull-back foliation $\tilde{\mathcal{L}}_\C$ is given by $\alpha_2|_{\tilde{M}^{*}_{\C}}=0$.
Similarly as before, the intersection of $\tilde{C}$ with the open subset $\tilde{U}_4$ is
$$
\tilde{C}\cap \tilde{U}_4=\{ w_2=Q(x_2,y_2)+Q(z_2,1)=0\}/\mathbb{Z}_{b},
$$
which is invariant by $\tilde{\mathcal{L}}_\C$ by (\ref{equa_2}), and
$$
\sing (\tilde{\mathcal{L}}_\C)\cap \tilde{U}_4=\{ w_2=Q(x_2,y_2)=Q(z_2,1)=0\}/\mathbb{Z}_b.
$$
Now, we focus in the chart $\tilde{U}_3$. In this open subset, the action of the group is given by
\begin{eqnarray*}
x_1&\mapsto& x_1,\\
y_1&\mapsto& e^{\frac{2bi\pi}{a}}y_1,\\
w_1&\mapsto & e^{\frac{2bi\pi}{a}}w_1.
\end{eqnarray*}
The exceptional divisor in this chart is given by
$$
\sing(D)\cap \tilde{U}_3 = \{ y_1=z_1=w_1=0\}/{\mathbb{Z}_a}
$$
and therefore the intersection of the singular set of $\tilde{\mathcal{L}}_\C$ with the singular set of the exceptional divisor is
$$
\sing(D)\cap \sing(\tilde{\mathcal{L}}_\C)\cap\tilde{
U}_3=\{ y_1=z_1=w_1=Q(1,0)=0\}/\mathbb{Z}_{a}.
$$
Due to the factorization of $Q$ given in (\ref{qh}), we investigated four cases.
\begin{itemize}
\item $m=n=1$. In this case, $Q(1,0)=0$ and, since $Q(1,w_1)$ is a complex polynomial in $w_1$, there exists another complex polynomial $\tilde{Q}$ such that $Q(1,w_1)=w_1\tilde{Q}(w_1)$ such that $\tilde{Q}(0)\neq 0$. Note that the power for $w_1$ may not be higher than one, because this would conflict with the fact that $n=1$ in the factorization of $Q$. Now, if $r$ is a root of $\tilde{Q}$, then $r\neq 0$ and therefore $(0,0,0,r)\in \sing(\tilde{\mathcal{L}}_\C)\cap\tilde{U}_3$ and $(0,0,0,r)\notin \sing(D)\cap \tilde{U}_3$. Hence, we get $\sing(D)\cap \tilde{U}_3\subsetneq \sing(\tilde{\mathcal{L}}_\C)\cap \tilde{U}_3$.
\item $m=0$, $n=1$. The same argument as the previous one holds in this case and therefore we have $\sing(D)\cap \tilde{U}_3\subsetneq \sing(\tilde{\mathcal{L}}_\C)\cap \tilde{U}_3$.
\item $m=1$, $n=0$. In this case, $Q(1,0)\neq 0$ and therefore $\sing(D)\cap \sing(\tilde{\mathcal{L}}_\C)\cap \tilde{U}_3=\emptyset$.
\item $m=0$. Same as before, we conclude that $\sing(D)\cap \sing(\tilde{\mathcal{L}}_\C)\cap \tilde{U}_3=\emptyset$.
\end{itemize}
We arrive to the same conclusions working in the chart $\tilde{U}_4$. In both cases we have shown that, either $\sing(D)\cap \sing(\tilde{\mathcal{L}}_\C)=\emptyset$ or that $\sing(D)\subsetneq \sing(\tilde{\mathcal{L}}_\C)$.
\par Consider the set $S:=\tilde{C}\setminus \sing (\tilde{\mathcal{L}}_\C)$. This set is a leaf of $\tilde{\mathcal{L}}_\C$. Let $q_0$ be a point in $S\setminus \sing(D)$ and a section $\Sigma$ transverse to $S$ passing through $q_0$. Working on the chart $\tilde{U}_3$, we may assume without loss of generality that $q_0=(1,0,0,0)$ and $\Sigma=\{ (1,0,t,0):\,\, t\in\C\}$. Let $G$ be the holonomy group of the leaf $S$ of $\tilde{\mathcal{L}}_\C$ in $\Sigma$. Recall that
$$
\sing(\tilde{\mathcal{L}}_\C)\cap \tilde{U}_3=\{ z_1 = Q(x_1,y_1)= Q(1,w_1)\}/\mathbb{Z}_{a}.
$$ 
This set splits into several connected components, separated in the following cases:

\begin{itemize}
\item \textit{$m=1$, $n=1$}. In this case, $$Q(x,y)=xy\prod_{\ell=1}^k Q_\ell(x,y),$$ where $Q_\ell(x,y)=(y^p-\lambda_{\ell} x^q)$, $\gcd(p,q)=1$ and $aq=pb=d$. The set $\sing(\tilde{\mathcal{L}}_\C)\cap \tilde{U}_3$ splits  as the union of the following connected components:
$$
C^{\ell}_{rs}=\{ z_1=Q_\ell(x_1,y_1)=w_1-\varepsilon^{(r)}_p(\lambda_s)=0\}/\mathbb{Z}_{a},
$$
$$
C^{x_1}_{rs}=\{ z_1=x_1=w_1-\varepsilon^{(r)}_p(\lambda_s)=0\}/\mathbb{Z}_{a},
$$
$$
C^{y_1}_{rs}=\{z_1=y_1=w_1-\varepsilon^{(r)}_p(\lambda_s)=0\}/\mathbb{Z}_{a},
$$
where $s,\ell\in\{1,...,k\}$ and $r\in\{1,...,p\}$ and for each $r$, $\varepsilon^{(r)}_p (\lambda_s)$ is an $p$-$th$ root of $\lambda_s$. According to \cite{zariski1932}, the fundamental group $\pi_1(S,q_0)$ may be written in terms of generators and its relations as
$$
\pi_1(S,q_0)=\left\langle \gamma_{\ell rs},\delta_{\ell rs},\xi_{rs},\tau_{rs}:\gamma_{\ell rs}^p=\delta^q_{\ell rs}\right\rangle _{\begin{tiny}\begin{array}{ll}\ell,s=1,...,k\\ r=1,...,p\end{array}\end{tiny}}
$$
where, for each $\ell,r,s$, the elements $\gamma_{\ell rs}$ and $\delta_{\ell rs}$ are loops around the connected component $C^{\ell}_{rs}$ of $\sing(\tilde{\mathcal{L}}_\C)\cap \tilde{U}_3$, $\xi_{rs}$ are loops around $C^{x_1}_{rs}$ and $\tau_{rs}$ a loop around $C^{y_1}_{rs}$. If $G$ is the holonomy group of the leaf $S$ of $\tilde{\mathcal{L}}_\C$ in the section $\Sigma$, then
$$
G=\displaystyle\left\langle f_{\ell rs},g_{\ell rs},h_{rs},k_{rs}\right\rangle _{\begin{tiny}\begin{array}{ll}\ell,s=1,...,k\\r=1,...,p\end{array}\end{tiny}}
$$
where $f_{\ell rs}$, $g_{\ell rs}$, $h_{rs}$ and $k_{rs}$ correspond to the equivalence classes of the loops $\gamma_{\ell rs}$, $\delta_{\ell rs}$, $\xi_{rs}$, $\tau_{rs}$ in $\pi_1 (S, q_0)$, respectively. Each one of these loops lifts up to $\Gamma_{\ell rs}(t)$, $\Delta_{\ell rs}(t)$, $\Xi_{rs}(t)$, $\Upsilon_{rs}(t)$, respectively, under the condition that each one of these belong on the leaves of $\tilde{\mathcal{L}}_{\C}$ and that this foliation is defined by $\alpha_1|_{M^{*}_{\C}}=0$ (see for instance (\ref{alpha_teo1})). We have the coefficients of the linear terms of the holonomy maps are given by 
\begin{eqnarray*}
f'_{\ell rs}(0)&=&e^{-\frac{2(1+qk)}{p+q+pqk}\pi i},\\
g'_{\ell rs}(0)&=& e^{-\frac{2}{q}\left(\frac{p+pqk}{p+q+pqk}\right)\pi i},\\
h'_{rs}(0)&=&1,\\
k'_{rs}(0)&=& e^{-2\left(\frac{1+pk}{p+q+pqk}\right)\pi i}.
\end{eqnarray*}
According to Lemma \ref{lemaarturo}, the foliation $\tilde{\mathcal{L}}_\C$ has a holomorphic non-constant first integral and the proof in this case is finished.

\item $m=0,n=1$. In this case, $$Q(x,y)=y\prod_{\ell=1}^k Q_\ell (x,t),$$ where $Q_\ell = (y^p-\lambda_\ell x^q)$, $\gcd(p,q)=1$ and $aq=pb=d$. The set $\sing(\tilde{\mathcal{L}}_\C)\cap \tilde{U}_3$ splits as the union of the following connected components:
$$
C^{\ell}_{rs}=\{z_1=Q_\ell (x_1,y_1)=w_1-\varepsilon^{(r)}_p(\lambda_s)=0\}/\mathbb{Z}_{a},
$$
$$
C^{y_1}_{rs}=\{ z_1 = y_1 = w_1-\varepsilon^{(r)}_p (\lambda_s)=0\}/\mathbb{Z}_{a},
$$
where $s,\ell\in\{1,...,k\}$, $r\in\{1,...,p\}$ and, for each $r$, $\varepsilon^{(r)}_p(\lambda_s)$ is a $p$-$th$ root of $\lambda_s$. The group $\pi_1(S,q_0)$ is written in terms of generators and its relations as
$$
\pi_1 (S,q_0)=\left\langle \gamma_{\ell rs},\delta_{\ell rs},\tau_{rs}:\gamma^p_{\ell rs}=\delta^q_{\ell rs}\right\rangle _{\begin{tiny}\begin{array}{ll}\ell,s=1,...,k\\r=1,...,p\end{array}\end{tiny}}
$$
where, for each $\ell,r,s$, $\gamma_{\ell rs}$ and $\delta_{\ell rs}$ are loops around $C^{\ell}_{rs}$ and $\tau_{rs}$ a loop around $C^{y_1}_{rs}$. If $G$ is the holonomy group of the leaf $S$ of $\tilde{\mathcal{L}}_\C$ in the section $\Sigma$ then
$$
G=\left\langle f_{\ell rs},g_{\ell rs},k_{rs}\right\rangle _{\begin{tiny}\begin{array}{ll}\ell,s=1,...,k\\r=1,...,p\end{array}\end{tiny}}
$$
where $f_{\ell rs}$, $g_{\ell rs}$ and $k_{rs}$ correspond to the equivalence classes of the loops $\gamma_{\ell rs}$, $\delta_{\ell rs}$, $\tau_{rs}$ in $\pi_1 (S, q_0)$, respectively. Each one of these loops lifts up to $\Gamma_{\ell rs}(t)$, $\Delta_{\ell rs}(t)$, $\Upsilon_{rs}(t)$, respectively, under the condition that each one of these belong on the leaves of $\tilde{\mathcal{L}}_{\C}$ and that this foliation is defined by $\alpha_1|_{M^{*}_{\C}}=0$ (see for instance (\ref{alpha_teo1})), we have the coefficients of the linear terms of the holonomy maps are given by 
\begin{eqnarray*}
f'_{\ell rs}(0)&=&e^{-\frac{2\pi i}{p}},\\
g'_{\ell rs}(0)&=&e^{-\frac{2 \pi i}{q}},\\
k'_{rs}(0)&=&1.
\end{eqnarray*}
 Using Lemma \ref{lemaarturo}, the proof in this case is finished.

\item $m=1,n=0$. In this case $$Q(x,y)=x\prod_{\ell=1}^k Q_\ell (x,t),$$ where $Q_\ell = (y^p-\lambda_\ell x^q)$, $\gcd(p,q)=1$ and $aq=pb=d$. The set $\sing(\tilde{\mathcal{L}}_\C)\cap \tilde{U}_3$ splits as the union of the following connected components:
$$
C^{\ell}_{rs}=\{z_1=Q_\ell (x_1,y_1)=w_1-\varepsilon^{(r)}_p(\lambda_s)=0\}/\mathbb{Z}_{a},
$$
$$
C^{y_1}_{rs}=\{ z_1 = x_1 = w_1-\varepsilon^{(r)}_p (\lambda_s)=0\}/\mathbb{Z}_{a},
$$
where $s,\ell\in\{1,...,k\}$, $r\in\{1,...,p\}$ and, for each $r$, $\varepsilon^{(r)}_p(\lambda_s)$ is a $p$-$th$ root of $\lambda_s$. The group $\pi_1(S,q_0)$ is written in terms of generators and its relations as
$$
\pi_1 (S,q_0)=\langle \gamma_{\ell rs},\delta_{\ell rs},\xi_{rs}:\gamma^p_{\ell rs}=\delta^q_{\ell rs}\rangle _{\begin{tiny}\begin{array}{ll}\ell,s=1,...,k\\r=1,...,p\end{array}\end{tiny}}
$$
where, for each $\ell,r,s$, $\gamma_{\ell rs}$ and $\delta_{\ell rs}$ are loops around $C^{\ell}_{rs}$ and $\xi_{rs}$ a loop around $C^{x_1}_{rs}$.  If $G$ is the holonomy group of the leaf $S$ of $\tilde{\mathcal{L}}_\C$ in the section $\Sigma$ then
$$
G=\left\langle f_{\ell rs},g_{\ell rs},h_{rs}\right\rangle _{\begin{tiny}\begin{array}{ll}\ell,s=1,...,k\\r=1,...,p\end{array}\end{tiny}}
$$
where $f_{\ell rs}$, $g_{\ell rs}$ and $h_{rs}$ correspond to the equivalence classes of the loops $\gamma_{\ell rs}$, $\delta_{\ell rs}$, $\xi_{rs}$ in $\pi_1 (S, q_0)$, respectively. Each one of these loops lifts up to $\Gamma_{\ell rs}(t)$, $\Delta_{\ell rs}(t)$, $\Xi_{rs}(t)$, respectively, under the condition that each one of these belong on the leaves of $\tilde{\mathcal{L}}_{\C}$ and that this foliation is defined by $\alpha_1|_{M^{*}_{\C}}=0$ (see for instance (\ref{alpha_teo1})), we have the coefficients of the linear terms of the holonomy maps are given by 
\begin{eqnarray*}
f'_{\ell rs}(0)&=&e^{-\frac{2\pi i}{q}},\\
g'_{\ell rs}(0)&=&e^{-\frac{2 \pi i}{p}},\\
k'_{rs}(0)&=&1,
\end{eqnarray*}
Again by Lemma \ref{lemaarturo}, the proof in this case is finished.

\item $m=0,n=0$. In this case, $Q(x,y)=\displaystyle\prod_{\ell=1}^k Q_\ell (x,t)$, where $Q_\ell = (y^p-\lambda_\ell x^q)$, $\gcd(p,q)=1$ and $aq=pb=d$. The set $\sing(\tilde{\mathcal{L}}_\C)\cap \tilde{U}_3$ splits as the union of the following connected components:
$$
C^{\ell}_{rs}=\{z_1=Q_\ell (x_1,y_1)=w_1-\varepsilon^{(r)}_p(\lambda_s)=0\}/\mathbb{Z}_{a},
$$
where $s,\ell\in\{1,...,k\}$, $r\in\{1,...,p\}$ and, for each $r$, $\varepsilon^{(r)}_p(\lambda_s)$ is a $p$-$th$ root of $\lambda_s$. The group $\pi_1(S,q_0)$ is written in terms of generators and its relations as
$$
\pi_1 (S,q_0)=\langle \gamma_{\ell rs},\delta_{\ell rs}:\gamma^p_{\ell rs}=\delta^q_{\ell rs}\rangle _{\begin{tiny}\begin{array}{ll}\ell,s=1,...,k\\r=1,...,p\end{array}\end{tiny}}
$$
where, for each $\ell,r,s$, $\gamma_{\ell rs}$ and $\delta_{\ell rs}$ are loops around $C^{\ell}_{rs}$.  If $G$ is the holonomy of the leaf $S$ of $\tilde{\mathcal{L}}_\C$ in the section $\Sigma$ then
$$
G=\langle f_{\ell rs},g_{\ell rs}\rangle _{\begin{tiny}\begin{array}{ll}\ell,s=1,...,k\\r=1,...,p\end{array}\end{tiny}}
$$
where $f_{\ell rs}$, $g_{\ell rs}$ correspond to the equivalence classes of the loops $\gamma_{\ell rs}$, $\delta_{\ell rs}$ in $\pi_1 (S, q_0)$, respectively. Each one of these loops lifts up to $\Gamma_{\ell rs}(t)$, $\Delta_{\ell rs}(t)$, respectively, under the condition that each one of these belong on the leaves of $\tilde{\mathcal{L}}_{\C}$ and that this foliation is defined by $\alpha_1|_{M^{*}_{\C}}=0$ (see for instance (\ref{alpha_teo1})), we have the coefficients of the linear terms of the holonomy maps are given by 
\begin{eqnarray*}
f'_{\ell rs}(0)&=&e^{-\frac{2\pi i}{q}},\\
g'_{\ell rs}(0)&=&e^{-\frac{2 \pi i}{p}}.
\end{eqnarray*}
Finally Lemma \ref{lemaarturo} implies that  $\tilde{\mathcal{L}}_\C$ has a holomorphic non-constant first integral.
\end{itemize}
\end{proof}
\subsection{Proof of Theorem \ref{teorema1}}

\par Note that Proposition \ref{proposition_integral} implies that the hypotheses of Theorem \ref{cerveaulinsneto1}, part $(b)$ are verified. Then there exists a germ of holomorphic foliation $\mathcal{F}_M$ with a non-constant holomorphic first integral $f\in\mathcal{O}_2$ such that $M=\{\mathcal{R}e(f)=0\}$. Without loss of generality, we can assume that $f$ is not a power in $\mathcal{O}_2$ and therefore so $\mathcal{R}e(f)$ is irreducible by Lemma \ref{lema1}. This implies $$\mathcal{R}e(f)=U\cdot F,$$ where $U\in\mathcal{A}_{n\mathbb{R}}$ and $U(0)\neq 0$. Since $F(z)=\mathcal{R}e(Q(z))+H(z,\bar{z})$ and $Q$ is a quasihomogeneous polynomial of quasihomogeneous degree $d$ with weights $(a,b)$, we can write $f$ as the decomposition
$$
f=\displaystyle \sum_{\ell\geq d} f_\ell,
$$
where each $f_\ell$ is a quasihomogeneous polynomial of quasihomogeneous degree $\ell$ with weights $(a,b)$ (see \cite[p. 193]{arnold1974}). If the power series of $U$ at $0\in\C^2$ is 
$$
U(z)=U(0)+\tilde{U}(z)=U(0)+\displaystyle\sum_{\substack{\mu_1+\mu_2\geq 1\\ \nu_1+\nu_2\geq 1}} c_{\mu_1\mu_2\nu_1\nu_2}z_1^{\mu_1} z_2^{\mu_2}\overline{z_1}^{\nu_1}\overline{z_2}^{\nu_2}
$$ then 
\begin{eqnarray*}
\mathcal{R}e(f)&=&(U(0)+\tilde{U})(\mathcal{R}e(Q)+H)\\
&=& U(0)\mathcal{R}e(Q)+\tilde{U}\mathcal{R}e(Q)+U(0)H+\tilde{U}H.
\end{eqnarray*}
We need to investigate what terms on the previous equality have quasihomogeneous degree $d$ with weights $(a,b)$, the sum of these terms will be equal to $\mathcal{R}e(f_d)$. Set $\tilde{H}=U(0)H+\tilde{U}H$, note that the quasihomogeneous terms of $\tilde{H}$ has order strictly greater than $d$. Writing 
 $$
 \mathcal{R}e(f)=\mathcal{R}e(f_d)+\sum_{\ell> d} \mathcal{R}e(f_\ell),
 $$
 we have, for all $\lambda\in\mathbb{C}^{*}$
\begin{eqnarray*}
 \mathcal{R}e(f(\lambda^az_1,\lambda^bz_2))&=&U(0)\mathcal{R}e(Q(\lambda^az_1,\lambda^bz_2))\\
                                                                          &    &+\tilde{U}(\lambda^az_1,\lambda^bz_2)\mathcal{R}e(Q(\lambda^az_1,\lambda^bz_2))+\tilde{H}(\lambda^az_1,\lambda^bz_2)\\
     &=&U(0)\mathcal{R}e(\lambda^d Q(z_1,z_2))\\
& &+\tilde{U}(\lambda^az_1,\lambda^bz_2)\mathcal{R}e(\lambda^d Q(z_1,z_2))+\tilde{H}(\lambda^az_1,\lambda^bz_2)\\
	 &=& U(0)\left(\dfrac{\lambda^dQ(z)+\overline{\lambda}^d\overline{Q(z)}}{2} \right)\\
	& &+\left( c_{1000}\lambda^az_1+c_{0100}\lambda^bz_2+c_{0010}\overline{\lambda}^a\overline{z_1}\right.\\
	& &\left.+c_{0001}\overline{\lambda}^b\overline{z_2}+... \right)\mathcal{R}e(\lambda^d Q(z))+\tilde{H}(\lambda^az_1,\lambda^bz_2)\\
&=& \underbrace{U(0)\left(\dfrac{\lambda^dQ(z)+\overline{\lambda}^d\overline{Q(z)}}{2} \right)}_{\text{generalized degree is }d}\\
& &+\underbrace{ c_{1000}\lambda^a\lambda^dz_1\mathcal{R}e(Q(z))+... +\tilde{H}(\lambda^az_1,\lambda^bz_2)}_{\text{generalized degree is greater than }d}
\end{eqnarray*}
which means that $f_d(z)=U(0)Q(z)$, hence $f(z)=U(0)Q(z)+\displaystyle\sum_{\ell>d}f_{\ell}$. Without any loss of generality we may assume that $U(0)=1$. In particular, $\mu(f,0)=\mu(Q,0)$, since $Q$ has an isolated singularity at the origin. According to Theorem \ref{lemaarnold}, there exists a germ of biholomorphism $\phi:(\mathbb{C}^{2},0)\rightarrow(\mathbb{C}^{2},0)$ such that
$$f\circ\phi^{-1}(z)=Q(z)+\displaystyle\sum_{j}c_je_j(z),$$ where $c_j\in\C$ and $e_j$ are elements of the monomial basis of $A_Q$ with $\deg(e_j)>d$. Hence $$\phi(M)=\left\{\mathcal{R}e\left(Q(z)+\sum_{j}c_je_j(z)\right)=0\right\}$$ and this finishes the proof of Theorem \ref{teorema1}. 

\begin{example}
Now we give an application of Theorem \ref{teorema1}. Consider the complex quasihomogeneous polynomial $$Q(x,y)=x^a+\lambda x^2 y^2+y^b\,\,\,\,\,\text{where}\,\,\,\,\,a\geq 4,\,\, b\geq 5,\,\, \lambda\neq 0.$$ 
We have $Q$ has isolated singularity at $0\in\mathbb{C}^2$ with $\mu(Q,0)=a+b+1$.
According to \cite[p. 33]{arnold1974}, every semiquasihomogeneous function $f$ with principal part $Q(x,y)$ is right equivalent to $Q(x,y)$.
Consequently, if we consider $F(x,y)=\mathcal{R}e(Q(x,y))+H(x,y)$ as a germ of real-analytic function at $0\in\mathbb{C}^2$ such that $M=\{F=0\}$ is Levi-flat then Theorem \ref{teorema1} implies that $M$ is biholomorphic to germ at $0\in\mathbb{C}^2$ of real-analytic Levi-flat hypersurface defined by 
$$M'=\{\mathcal{R}e(x^a+\lambda x^2 y^2+y^b)=0\}\,\,\,\,\,\text{where}\,\,\,\,\,a\geq 4,\,\, b\geq 5,\,\, \lambda\neq 0.$$ 
\end{example}

\section{Isochore normal forms for holomorphic functions}

Let $f\in\mathcal{O}_n$ be a germ of holomorphic function with an isolated singularity at $0\in\C^n$  such that its Hessian form 
$$
h:= \sum_{1\leq i,j\leq n} \dfrac{\partial^2 f(0)}{\partial z_i\partial z_j}z_iz_j
$$
is non-degenerate. The classical Morse's lemma asserts that $f$ is right equivalent to $h$.
\par Let $\omega = a(z)dz_1\wedge ...\wedge dz_n$, $a(0)\neq 0$ be a holomorphic volume form on a coordinate system $(z_1,...,z_n)$ on an open set around $0\in\C^n$. A coordinate system $(x_1,\ldots,x_n)$ is 
\emph{isochore} or \emph{volume preserve}, if $\omega$ can be written as $dx_1\wedge\ldots\wedge dx_n$ on these coordinates. Then, we say that a biholomorphism $\phi:(\C^n,0)\rightarrow (\C^n,0)$ is \textit{isochore} or \textit{volume-preserving} if the coordinate system induced by it is isochore.
\par In 1977, J. Vey \cite{vey1977} has posed the following question: It is possible to find a coordinate system isochore such that $f$ is right equivalent to $h$?. 
Vey answered negatively to question and proved the following result. 
\begin{lemma}[Vey \cite{vey1977}]\label{vey}
Let $f\in\mathcal{O}_n$, $n\geq 2$, with isolated singularity at $0\in\C^n$ such that its Hessian form $h$ is non-degenerate. Then there exists a germ of a volume-preserving biholomorphism $\phi:(\mathbb{C}^n,0)\to(\mathbb{C}^n,0)$ and a germ of an automorphism $\psi\in\mathcal{O}_1$, with $\psi(0)=0$, such that
$$
f\circ \phi^{-1}=\psi\circ h,\,\,\,\,\,\,\,\psi(t)=t+c_2t^2+c_3t^3+\ldots
$$
The function $\psi$ is uniquely determined by $f$ up to a sign.
\end{lemma}
\par This result was also proved by J-P Fran\c{c}oise \cite{francoise1978}. The approach used by Fran\c{c}oise was later generalized by A. Szawlowski \cite{szawlowski2012} to study of complex quasihomogeneous polynomials and to the germ of a holomorphic function that is right equivalent to the product of coordinates $z_1\cdot\ldots\cdot z_n$, as stated by the following theorem.

\begin{theorem}[Szawlowski \cite{szawlowski2012}]\label{szawlowski}
Let $f\in\mathcal{O}_n$, $n\geq 2$ be a germ of holomorphic function that is right equivalent to the product of all coordinates: $f\sim_R z_1\cdot\ldots\cdot z_n$. Then there exists a germ of a volume-preserving biholomorphism $\Phi:(\mathbb{C}^n,0)\to(\mathbb{C}^n,0)$ and a germ of an automorphism $\Psi\in\mathcal{O}_1$, with $\Psi(0)=0$, such that

$$
f\circ\Phi(z) = \Psi (z_1\cdot ...\cdot z_n).
$$
The function $\Psi$ is uniquely determined by $f$ up to a sign.
\end{theorem}
\par Note that the above normal form for $f$ is a germ of holomorphic function whose singular set is of positive dimension (non-isolated singularity). In general, normal forms of germs of functions with non-isolated singularities are very difficult of find, even for a change of coordinates non-isochore.   
\section{Theorem \ref{teorema2}}
\par To prove Theorem \ref{teorema2} we use the following result proved in \cite{perez2011}, although it is not stated as a separate theorem. We restate it here for completeness.  
\begin{theorem}[Fern\'andez-P\'erez \cite{perez2011}]\label{perez1}
Let $M=\{F=0\}$ be a germ of an irreducible singular real-analytic  Levi-flat hypersurface at $0\in\mathbb{C}^n$, $n\geq 2$, such that
\begin{enumerate}
\item $F(z)=\mathcal{R}e(P(z))+H(z,\bar{z})$,
\item $P$ is a complex homogeneous polynomial of degree $k$ with an isolated singularity at $0\in\mathbb{C}^{n}$,
\item $j_0^k(H)=0$ and $H(z,\bar{z})=\overline{H(\bar{z},z)}$.
\end{enumerate}
Then there exists a germ at $0\in\mathbb{C}^n$ of holomorphic codimension-one foliation $\mathcal{F}_M$ tangent to $M$. Moreover, the foliation $\mathcal{F}_M$ has a non-constant holomorphic first integral $f(z)=P(z)+O(|z|^{k+1})$, and $M=\{\mathcal{R}e(f)=0\}$.
\end{theorem}
\subsection{Proof of Theorem \ref{teorema2}}
Let $M=\{F=0\}$ be a germ at $0\in\mathbb{C}^n$, $n\geq 2$, of an irreducible real-analytic Levi-flat hypersurface  such that
$$
F(z)=\mathcal{R}e (z_1^2+...+z_n^2)+H(z,\bar{z}),
$$ 
where $j_0^2(H)=0$, $H(z,\bar{z})=\overline{H(\bar{z},z)}$. Since $P(z_1,\ldots,z_n)=z^{2}_1+\ldots+z_n^2$ is a complex homogeneous polynomial of degree 2,  we can apply Theorem \ref{perez1}, so that there exists $f\in\mathcal{O}_n$ such that 
$f(z)=z_1^2+\ldots+z^2_n+O(|z|^3)$ and $M=\{\mathcal{R}e(f)=0\}$. 
On the other hand, applying Lemma \ref{vey} to $f$, there exists a volume-preserving $\phi:(\mathbb{C}^n,0)\to(\mathbb{C}^n,0)$ and an automorphism $\psi_1\in\mathcal{O}_1$, with $\psi_1(0)=0$, such that
$$
f\circ \phi^{-1}=\psi_1 (2P),\,\,\,\,\,\,\,\psi_1(t)=t+c_2t^2+c_3t^3+\ldots
$$
Taking $\psi:=\psi_1(t/2)\in\mathcal{O}_1$, we have $f\circ\phi^{-1}=\psi\circ P$. Finally, $\phi(M)=\{\mathcal{R}e(\psi (z^{2}_1+\ldots+z_n^2))=0\}$ and the proof of Theorem \ref{teorema2} ends. 

\section{Proof of Theorem \ref{teorema3} and Corollary \ref{coro1}}
Here we will use the same idea of the proof of Theorem \ref{teorema1}. First of all, note that, in dimension two, under the change of variables $z_1=y+ix$, $z_2=y-ix$, and we have $z_1z_2= x^2+y^2$ and then Theorem \ref{teorema3} follows from Theorem \ref{teorema2}, because the singular set of $M_{\C}$ is the origin of $\C^4$. Therefore, we only consider the case $n\geq 3$.
\begin{proposition}\label{propo_teorema3}
Let $M$ be a germ of a singular real-analytic Levi-flat hypersurface at $0\in\mathbb{C}^n$, $n\geq 3$, satisfying the hypotheses of Theorem \ref{teorema3}. Then $\mathcal{L}_\C$ has a non-constant holomorphic first integral.
\end{proposition}
\begin{proof}
\par Let $M$ be as in Theorem \ref{teorema3}. Then, $M$ is given by $\{F=0\}$ where $$F(z)=\mathcal{R}e(z_1\cdots z_n) + H(z_1,...,z_n),$$ and $j^n_0(H)=0$. Its complexification is
\begin{equation}\label{complex_1}
F_\C(z,w)=\dfrac{1}{2}(z_1\cdots z_n) + \frac{1}{2}(w_1\cdots w_n)+H_\C(z,w),
\end{equation}
and therefore $M_\C =\{F_\C=0\}\subset (\C^{2n},0)$. By hypotheses, $\sing(M_\C)$ is the union  of the sets 
$$
V_{ijk\ell}=\{ z_i=z_j=w_k=w_{\ell}=0\},\,\,\,\,\,1\leq i< j\leq n,\,\,\,\,1\leq k< \ell\leq n.$$
Since $V_{ijk\ell}$ has complex dimension $2n-4$, then the algebraic dimension of $\sing(M)$ is $2n-4$. 
\par On the other hand, it follows from Remark \ref{remark_fol} that $\mathcal{L}_\C$ is given by $\alpha\vert_{M^*_\C}=0$, where 
$$\alpha=\sum^{n}_{i=1} \dfrac{\partial F_\C}{\partial z_i}dz_i.$$
Using (\ref{complex_1}) we can write $\alpha$ in coordinates $(r_1,\ldots,r_n)\in\mathbb{C}^n$ as 
$$
\alpha = \frac{1}{2}\displaystyle\sum^n_{i=1} \left( r_1\cdots\widehat{r}_i\cdots r_n +\dfrac{\partial R}{\partial r_i} \right)dr_i,
$$
where $\dfrac{\partial R}{\partial r_i}=2\dfrac{\partial H_{\C}}{\partial r_i}$ for all $i=1\ldots,n$. Then we can consider that $\mathcal{L}_{\C}$ is defined by $\tilde{\alpha}|_{M^{*}_{\C}}=0$, where   
 $$\tilde{\alpha}= \displaystyle\sum^n_{i=1} \left( r_1\cdots\widehat{r}_i\cdots r_n +\dfrac{\partial R}{\partial r_i} \right)dr_i.$$
\par Let us prove that $\mathcal{L}_\C$ has a non-constant holomorphic first integral. We start with the blow-up $\pi_1$ at $0\in\C^n$ with exceptional divisor $D_1\cong\mathbb{P}^{2n-1}$. Let $[Z:Y]=[Z_1:\ldots:Z_n:Y_1:\ldots:Y_n]$ be the homogeneous coordinates of $D_1$. The intersection of $\tilde{M}_\C=\pi_1^*(M_\C)$ with the divisor $D_1$ is the algebraic hypersurface
$$
Q_1:=\tilde{M}_\C\cap D_1 = \{[Z:Y]\in\mathbb{P}^{2n-1}:\,\,Z_1\cdots Z_n+Y_1\cdots Y_n = 0\}.
$$
In the chart $(W,(r,\ell)=(r_1,...,r_n,\ell_1,...,\ell_n))$ of $\tilde{\C}^{2n}$ where
$$
\pi_1(r,\ell)=(\ell_1r_1,...,\ell_1r_2,...,\ell_1r_n,\ell_1,\ell_1\ell_2,...,\ell_1\ell_n).
$$
Then
$$
\begin{array}{rcl}
\tilde{F}_\C(r,\ell)&=&F_\C\circ \pi_1(r,\ell)=\ell_1^nr_1\cdots r_n+\ell_1^n\ell_2\cdots \ell_n+R(\pi_1(r,\ell))\\
&=&\ell_1^n(r_1\cdots r_n + \ell_2\cdots \ell_n + \ell_1R_1(r,\ell)),
\end{array}
$$
where $R_1(r,\ell)=R(\pi_1(r,\ell))/\ell_1^{n+1}$. Therefore
$$
\tilde{M}_\C\cap W = \{r_1\cdots r_n + \ell_2\cdots \ell_n + R_1(r,\ell) = 0\},
$$
and
$$
Q_1\cap W = \{\ell_1 = r_1\cdots r_n + \ell_2\cdots \ell_n = 0\}.
$$
On the other hand, the pull-back of $\tilde{\alpha}$ by $\pi_1$ is
$$
\begin{array}{rcl}
\pi^*_1(\tilde{\alpha}) & = & \displaystyle\sum_{i=1}^n \ell_1^{n-1}\left( r_1\dots \widehat{r}_i\cdots r_n\right) d(\ell_1r_i)+\theta\\
\\
&=&\ell_1^{n-1}\left( \displaystyle\sum^n_{i=1}\ell_1r_1\cdots \widehat{r}_i\cdots r_ndr_i + nr_1\cdots r_nd\ell_1+\ell_1\theta_1\right),
\end{array}
$$
where $\theta_1=\theta/\ell_1^n$. In the chart $W$, the exceptional divisor is written as $D_1=\{\ell_1=0\}$ and $\tilde{\mathcal{L}}_\C$ is given by $\alpha_1|_{\tilde{M}^{*}_{\C}}=0$, where 
$$
\alpha_1=\sum^n_{i=1}\ell_1r_1\cdots \widehat{r}_i\cdots r_ndr_i+nr_1\cdots r_nd\ell_1+\ell_1\theta_1.
$$
Note that $\tilde{M}_\C\cap D_1 $ is invariant  by $\tilde{\mathcal{L}}_\C$ and moreover 
$$
\sing(\tilde{M}_\C)\cap W = \bigcup_{i,j,k,s}
W_{i,j,k,s},
$$
where
$$
W_{i,j,k,s}:=\{ r_i=r_j=\ell_k=\ell_s=0\}_{\footnotesize{
\begin{array}{r}
1\leq i,j,k,s \leq n\\
\end{array}}}\,\,\,\text{where}\,\,\,\,i\neq j,k\neq s\,\,\,\,\text{and}\,\,\,
k\neq 1,s\neq 1.
$$
Consider the irreducible component $W_{1,2,2,3}$ of $\sing(\tilde{M}_\C)\cap W$. We make a blow-up along this component; the process of desingularization around the other components of $\sing(\tilde{M}_\C)\cap W$ are similarly obtained by exchanging coordinates. Let $E$ be the exceptional divisor of $\pi_{\ell}:\tilde{\C}^{2n}\rightarrow \C^{2n}$. Let $\tilde{\tilde{M}}_\C$ be the strict transform of $\tilde{M}_\C$ and $\tilde{\tilde{\mathcal{L}}}_\C$ be the pull-back of $\tilde{\mathcal{L}}_\C$ by $\pi_{\ell}$ respectively. Let $U$ be an open subset with coordinates $(x_1,...,x_{2n})$ where the blow-up is 
$$
\pi_{\ell}(x_1,...,x_{2n})=(x_1x_{n+3},x_2x_{n+3},x_3,...,x_n,x_{n+1},x_{n+2}x_{n+3},x_{n+3},x_{n+4},...,x_{2n}),
$$
we have
$$
\tilde{\tilde{F}}_\C=\tilde{F}_\C\circ\pi_{\ell}=x_{n+1}^nx_{n+3}^2(x_1\cdots x_n+x_{n+1}x_{n+2}x_{n+4}\cdots x_{2n}+x_{n+1}x_{n+3}R_2),
$$
where $R_2=R_1(\pi_{\ell} (x_1,...,x_{2n}))/x_{n+3}^3$. Therefore
$$
\tilde{\tilde{M}}_\C\cap U = \{x_1\cdots x_n+x_{n+2}x_{n+4}\cdots x_{2n}+x_{n+1}x_{n+2}R_2=0\}
$$
hence
$$ \tilde{\tilde{M}}_{\C}\cap E\cap U = \{x_{n+1}=x_{n+3}=x_1\cdots x_n+x_{n+2}x_{n+4}\cdots x_{2n} = 0\}.
$$
The pull-back of $\alpha_1$ by $\pi_{\ell}$ is
\begin{eqnarray*}
\pi^*_{\ell}(\alpha_1 )& = & x_{n+3}\left(x_2\cdots x_nx_{n+1}x_{n+3}dx_1+x_1x_3\cdots x_n x_{n+1}x_{n+3}dx_2+\right.\\
&  & \displaystyle\sum_{i=3}^n \dfrac{x_1\cdots x_nx_{n+1}x_{n+3}}{x_i}dx_i +\\
& &\left.nx_1\cdots x_nx_{n+3}dx_{n+1}+2x_1x_2\cdots x_nx_{n+1}dx_{n+3}+x_{n+1}x_{n+3}\theta_2\right),
\end{eqnarray*}
where $\theta_2=\theta_1/x_{n+3}^2$. In the chart $U$, the exceptional divisor is written as 
$$
D=D_1\cup D_2=\{x_{n+1}=0\}\cup \{x_{n+3}=0\}
$$
and $\tilde{\tilde{{\mathcal{L}}}}_\C$ is given by $\alpha_2|_{\tilde{\tilde{M}}^{*}_\C}=0$, where 
\begin{eqnarray}\label{alpha2}
\alpha_2 &= &x_2\cdots x_nx_{n+1}x_{n+3}dx_1+x_1x_3\cdots x_n x_{n+1}x_{n+3}dx_2+\nonumber\\
& &\sum_{i=3}^n \dfrac{x_1\cdots x_nx_{n+1}x_{n+3}}{x_i}dx_i+ \\
& &nx_1\cdots x_nx_{n+3}dx_{n+1}+2x_1x_2\cdots x_nx_{n+1}dx_{n+3}+x_{n+1}x_{n+3}\theta_2,\nonumber
\end{eqnarray}
which allows us to conclude that $\tilde{\tilde{M}}_\C\cap D$ is invariant by $\tilde{\tilde{{\mathcal{L}}}}_\C$. The singularities of the foliation $\tilde{\tilde{{\mathcal{L}}}}_\C$ on the exceptional divisor in this chart are given by
$$
\sing (\tilde{\tilde{{\mathcal{L}}}}_\C)\cap D\cap U = \left\{ x_{n+1}=x_{n+3}=x_1\cdots x_n = x_{n+2}x_{n+4}\cdots x_{2n}=0\right\}.
$$
If we define $\mathcal{C}_{i,n+j}=\{ x_{n+1}=x_{n+3}=x_i=x_{n+j}=0\}\cong\C^{2(n-2)}$, then we can write
$$
\sing(\tilde{\tilde{{\mathcal{L}}}}_\C)\cap D\cap U = \displaystyle\bigcup_{\begin{tiny}
\begin{array}{c}
1\leq i,j\leq n\\j\neq 1,3\end{array}\end{tiny}} \mathcal{C}_{i,n+j}.
$$
Since $ \tilde{\tilde{M}}_\C\cap D$ is invariant by $\tilde{\tilde{{\mathcal{L}}}}_\C$, then
$$
S:= (\tilde{\tilde{M}}_\C\cap D)\setminus \sing(\tilde{\tilde{{\mathcal{L}}}}_\C)
$$
is a leaf of $\tilde{\tilde{{\mathcal{L}}}}_\C$. Let $G$ be its holonomy group, and $p_0\in S$ given by
$$
p_0=(x_1,...,x_n,x_{n+1},x_{n+1},x_{n+3},x_{n+4},...,x_{2n})=(1,...,1,0,-1,0,1,...,1).
$$
Take $\Sigma$ the transversal section through $p_0$ given by
$$
\Sigma = \{ (1,...,1,\lambda,-1,\lambda,1,...,1):\,\, \lambda\in\C\}.
$$
Let $\delta_{i,j}(\theta)$ be a loop around $\mathcal{C}_{i,n+j}$, for $1\leq i\leq n$ and $4\leq j \leq n$, and $\delta_{i,2}(\theta)$ a loop around $C_{i,n+2}$, $1\leq i\leq n$ with $\theta\in [0,1]$. Each one of these loops lifts up to $\Gamma _{i,j}(\lambda,\theta)$ and $\Gamma _{i,2}(\lambda,\theta)$, respectively, such that $\Gamma_{i,j}(0,\theta)=0$, $\Gamma_{i,j}(\lambda,0)=\lambda$ and $\Gamma_{i,j}(\lambda,\theta)=\displaystyle\sum_{k=1}^\infty \delta^{i,j}_k(\theta)\lambda^k$, for $i=1,...,n$ and $j=2,4,5,...,n$. The holonomy map with respect to these loops are
$$
h_{\delta_{i,j}}(\lambda)=\Gamma_{i,j}(\lambda,1).
$$
Using the expression of $\alpha_2$ given in (\ref{alpha2}), we get
$$
h'_{\delta_{i,j}}(0)=e^{-\frac{2\pi i}{n+2}}, \text{ for }i=1,...,n\text{ and } j=2,4,5,...,n.
$$
It follows from Lemma \ref{lemaarturo} that $\mathcal{L}_\C$ has a non-constant holomorphic first integral.
\end{proof}
\subsection{Proof of Theorem \ref{teorema3}}
\par Note that Proposition \ref{propo_teorema3} implies that the hypotheses of Theorem \ref{cerveaulinsneto1}, part $(b)$ are verified. Then we get $f\in\mathcal{O}_n$ such that the foliation $\mathcal{F}$ given by $df=0$ is tangent to $M$ and $M=\{\mathcal{R}e(f)=0\}$.
Without loss of generality we may assume that $f$ is not a power in $\mathcal{O}_n$ and therefore $\mathcal{R}e(f)$ is irreducible in $\mathcal{A}_{n\mathbb{R}}$. We must have that
 $Re(f)=U\cdot F$ where $U\in\mathcal{A}_{n\mathbb{R}}$, $U(0)\neq 0$. 
If the Taylor expansion of $f$ at $0\in\C^n$ is $$f=\displaystyle\sum_{j\geq n} f_j,$$ where $f_j$ is a homogeneous polynomial of degree $j$, then 
$$
\mathcal{R}e(f_n) = j_0^n(\mathcal{R}e(f))=j_0^n(U\cdot F)=U(0)\mathcal{R}e(z_1\cdots z_n),
$$
which means $f_n(z)=U(0) z_1\cdots z_n$. We can assume that $U(0)=1$ and therefore $$f(z)=z_1\cdots z_n+O(|z|^{n+1}).$$ 
This finishes the proof of Theorem \ref{teorema3}. 
\subsection{Proof of Corollary \ref{coro1}}
If we assume that $f(z)\sim_R z_1\cdots z_n$, it follows from Theorem \ref{szawlowski} that there exists a germ of a volume-preserving biholomorphism $\Phi:(\C^n,0)\rightarrow (\C^n,0)$ and a germ of an automorphism $\Psi:(\C,0)\rightarrow (\C,0)$, such that
$$
f\circ \Phi^{-1}(z) = \Psi(z_1\cdots z_n).
$$
Hence
$$
\Phi(M)=\{\mathcal{R}e(\Psi(z_1\cdots z_n))=0\}.
$$
 This finishes the proof of Corollary \ref{coro1}.

\vskip 0.2 in

\noindent{\it\bf Acknowledgments.--}
The authors gratefully acknowledges the many helpful suggestions of Rog\'erio Mol (UFMG) during the preparation of the paper. 


\end{document}